\documentclass[12pt, english]{amsart}

\usepackage{times}

\usepackage{babel}
\usepackage{amsopn}
\usepackage{amsmath,amsthm,amssymb}

\usepackage[hypertex]{hyperref} 

\newcommand{\nc}{\newcommand}

\newcommand{\fa}{{\mathfrak a}}

\newcommand{\ff}{{\mathfrak f}}
\newcommand{\fg}{{\mathfrak g}}
\newcommand{\fh}{{\mathfrak h}}

\newcommand{\fk}{{\mathfrak k}}
\newcommand{\fl}{{\mathfrak l}}

\nc{\fo}{{\mathfrak o}}

\newcommand{\fs}{{\mathfrak s}}
\newcommand{\fu}{{\mathfrak u}}
\newcommand{\fv}{{\mathfrak v}}

\newcommand{\fz}{{\mathfrak z}}
\newcommand{\cc}{{\mathbb C}}
\nc{\aff}{\fa\ff\ff}

\nc{\nn}{{\mathbb N}}
\newcommand{\rr}{{\mathbb R}}
\newcommand{\RR}{{\mathbb R}}
\nc{\qq}{\mathbb Q} \nc{\zz}{{\mathbb Z}}
\nc{\ad}{\operatorname{ad}}

\newtheorem{teo}{Theorem}[section]
\newtheorem{lem}[teo]{Lemma}
\newtheorem{prop}[teo]{Proposition}
\newtheorem{corol}[teo]{Corollary}
\newtheorem{rem}[teo]{Remark}
\newtheorem{example}{Example}

\nc{\lp}{\langle} \nc{\rp}{\rangle} \nc{\inc}{\hookrightarrow}

\title[ ]{\bf Abelian hermitian geometry }
\author{A. Andrada}
\date{}
\address{FaMAF-CIEM, Universidad Nacional de C\'{o}rdoba, Ciudad Universitaria, 5000 C\'{o}rdoba, Argentina}
\email{andrada@famaf.unc.edu.ar}
\thanks{}

\author{M. L. Barberis}
\email{barberis@famaf.unc.edu.ar}
\author{I. G. Dotti}
\email{idotti@famaf.unc.edu.ar}

\thanks{The authors were partially supported by
CONICET, ANPCyT and SECyT-UNC (Argentina).}

\begin{document}

\begin{abstract} 
We study the structure of Lie groups admitting left invariant abelian complex structures in terms of commutative associative algebras. If, in addition, the Lie group is equipped with a left invariant Hermitian structure, it turns out that such a Hermitian structure is K\"ahler if and only if the Lie group is the direct product of several copies of the real hyperbolic plane by a euclidean factor. Moreover, we show that if a left invariant Hermitian metric on 
a Lie group  with an abelian complex structure has flat first canonical connection, then the Lie group is abelian.
\end{abstract}

\maketitle

\section{Introduction}

An {\it abelian} complex structure on a real Lie algebra $\fg$ is
an endomorphism $J$ of $\fg$ satisfying
\[ J^2=-I, \hspace{1.5cm}
[Jx,Jy]=[x,y], \; \;\; \forall\, x,y \in \fg, \label{abel} \] or
equivalently, the $i$-eigenspace of $J$ in the complexification
$\fg ^{\cc}$ of $\fg$ is an abelian subalgebra of $\fg ^{\cc}$.
 If $G$
is a Lie group with Lie algebra $\fg$ these conditions imply the
vanishing of the Nijenhuis tensor on the invariant almost complex
manifold $(G,J)$, that is, $J$ is integrable on $G$. If $\Gamma
\subset G$ is any discrete subgroup of $G$ then the induced  $J$
on $\Gamma \backslash G$ will be also called invariant.

Our interest arises from properties of the complex manifolds
obtained by considering this class of complex structures on Lie
algebras. For instance, an abelian hypercomplex structure on
$\fg$, that is, a pair of anticommuting abelian complex
structures, gives rise to an invariant weak HKT structure on $G$
(see \cite{cqg} and \cite{gp}).
%
It was proved in \cite{BDV} that the converse of this result holds
for nilmanifolds, that is,  the hypercomplex structure associated
to a nilmanifold carrying an HKT structure is necessarily abelian.
 Also, in \cite{c,cfp,mpps} it was shown that invariant abelian complex structures
on nilmanifolds have a locally complete family of deformations
consisting of invariant complex structures.

In the first part of the article, we consider a distinguished class of Lie algebras admitting 
abelian complex structures, given by abelian double products.  
An abelian double product is a Lie algebra $\fg$
together with an abelian complex structure $J$ and a decomposition
 $\fg = \fu \oplus J \fu$, where both, $\fu$ and 
$J\fu$ are abelian subalgebras. The structure of these Lie algebras can be described in terms of 
a pair of commutative associative algebras satisfying a
compatibility condition (see \cite{AD,AS}). When $\fu$ is an ideal of
$\fg$ one obtains $\aff ({\mathcal A})$, where ${\mathcal A}$ is a commutative
associative algebra, a class of Lie algebras considered in
\cite{ba-do}. On the other hand, when $\fg$ is a Lie algebra with an abelian complex structure $J$, we prove in Theorem \ref{aff} that if $\fg$ decomposes as $\fg= \fu+J\fu$, with $\fu$ an abelian 
subalgebra, then $\fg$ is an abelian double product. Moreover, if $\fu$ is an abelian ideal and $\fg'\cap J\fg'=\{0\}$, then $\fg$ is an affine Lie algebra $\aff ({\mathcal A})$, for some commutative
associative algebra ${\mathcal A}$.

In the second part, the Hermitian geometry of Lie groups equipped with abelian complex structures is studied. First, and related with results mentioned above, we obtain in Theorem~\ref{kahler-abelian} and Corollary~\ref{kahler-group} that such a Lie group satisfying the K\"ahler condition is the product of hyperbolic spaces and a flat factor. Second, we investigate the first canonical Hermitian connection on Lie groups associated to a left invariant Hermitian structure with  abelian complex structure. This connection, which was introduced by Lichnerowicz in \cite{Li} and appears in the set of canonical Hermitian connections of Gauduchon \cite{Ga}, has torsion tensor of type $(1,1)$ with respect to the complex structure. In \cite{ABD}, the notion of abelian complex structure on Lie groups was extended to parallelizable manifolds. This generalization amounts to the existence of a complex connection on the complex manifold with trivial holonomy and torsion of type $(1,1)$ with respect to the complex structure. This motivates our study of the flatness of the first canonical Hermitian connection in the abelian case. We prove that if the first canonical connection associated to a left invariant Hermitian structure with  abelian complex structure is flat, then the Lie group is abelian.

\section{Preliminaries}\label{prelim}

In this section we will recall  basic definitions and known
results which will be used throughout the paper.

\subsection{Connections on Lie algebras} Let $G$ be a Lie group with Lie algebra $\fg$ and
suppose that $G$ admits a left invariant affine connection
$\nabla$, i.e., each left translation is an affine transformation
of $G$. In this case, if $X,Y$ are two left invariant vector
fields on $G$ then $\nabla_XY$ is also left invariant. Moreover,
there is a one--one correspondence between the set of left invariant connections on $G$ and the set of $\fg$-valued
bilinear forms $\fg \times \fg \to \fg$ (see \cite[p.102]{He}). Therefore,
such a bilinear form $\nabla:\fg\times \fg\to \fg$ will be called
a connection on $\fg$. The torsion $T$ and the curvature of
$\nabla$ are defined as follows:
\begin{align*}
T(x,y) & =\nabla_xy-\nabla_yx-[x,y]\\
R(x,y) &=\nabla_x\nabla_y-\nabla_y\nabla_x-\nabla_{[x,y]},
\end{align*}
for any $x,y\in\fg$. The connection $\nabla$ is torsion-free when
$T=0$, and $\nabla$ is flat when $R=0$. We note that in the flat
case, $\nabla:\fg\to \operatorname{End}(\fg)$ is a representation
of $\fg$.

The connection $\nabla$ on $\fg$ defined by $\nabla=0$ is known as
the $(-)$-connection. Its torsion $T$ is given by $T(x,y)=-[x,y]$
for all $x,y\in\fg$ and clearly $R=0$.

\subsection{Complex structures on Lie algebras}\label{complex}

A complex structure on a Lie algebra $\mathfrak g$ is an
endomorphism $J$ of $\mathfrak g$ satisfying $J^2= -I$ together
with the vanishing of the Nijenhuis bilinear form with values in
$\mathfrak g$,
\begin{equation}\label{nijen} N(x,y)= [Jx,Jy]-J[Jx,y]- J[x,Jy] -[x,y],\end{equation}
where $x,y\in \mathfrak g$.

If $G$ is a Lie group with Lie algebra $\fg$, by left translating
the endomorphism $J$ we obtain a  complex manifold $(G,J)$ such
that left translations are holomorphic maps. A complex structure
of this kind is called left invariant. We point out that $(G,J)$
is not necessarily a complex Lie group since right translations
are not in general holomorphic.

Consider a Lie algebra $\fg$ equipped with a complex structure
$J$. We denote with $ \fg'$ the commutator ideal $[\fg ,\fg]$ of
$\fg$ and with $\fg'_J$ the $J$-stable ideal of $\fg$ defined by
$\fg'_J:=\fg'+J\fg'$.

\smallskip

We will consider connections on $\fg$ compatible with the complex
structure $J$, and therefore we will say that a connection
$\nabla$ on $\fg$ is complex when $\nabla J=0$, that is,
$\nabla_xJ=J\nabla_x$ for any $x\in\fg$.

The torsion $T$ of a connection $\nabla$ on $\fg$ is said to be of
type $(1,1)$ if $T(Jx,Jy)=T(x,y)$ for all  $x,y$ on $\fg$.

\medskip

\subsection{Abelian complex structures} \label{secaff}

An {\it abelian} complex structure on $\fg$ is an endomorphism $J$
of $\fg$ satisfying
\begin{equation} J^2=-I, \hspace{1.5cm} [Jx,Jy]=[x,y], \; \;\;
\forall x,y \in \fg. \label{abel1}
\end{equation}
It follows that condition \eqref{abel1} is a particular case of
\eqref{nijen}. We note that the $(-)$-connection on $\fg$ has
torsion of type $(1,1)$ if and only if $J$ is abelian.

Abelian complex structures 
have been studied by several authors \cite{cfp, cfu, mpps, V}. 
A complete classification of the Lie algebras admitting abelian
complex structures is known up to dimension $6$ (see \cite{ABD1})
and there are structure results for arbitrary dimensions (see
\cite{ba-do}).

\medskip

The next result states some properties of Lie algebras admitting
abelian complex structures that will be used in forthcoming
sections.

\begin{lem}\label{cap}
Let $J$ be an abelian complex structure on the Lie algebra $\fg$.
Then:
\begin{enumerate}
\item[(i)] the center $\fz$ of $\fg$ is $J$-stable;
\item[(ii)] for any $x\in\fg$, $\ad_{Jx}=-\ad_xJ$;
\item[(iii)]  $\fg'$ is abelian, equivalently, $\fg$ is $2$-step
solvable;
\item[(iv)]  $J\fg'$ is an abelian subalgebra;
\item[(v)] $\fg'\cap J\fg'\subseteq \fz ( \fg '_J)$.
\end{enumerate}
\end{lem}

\begin{proof}
(i) and (ii) are straightforward. (iii) is a consequence of
results in  {\cite{Pe}}. Using (iii) and
the fact that $J$ is abelian, (iv) follows.

If $x\in\fg'\cap J\fg'$, then (iii) and (iv) imply that
$[x,\fg']=0 =[x,J\fg']$, thus $x\in\fz( \fg '_J)$ and (v) holds.
\end{proof}

\medskip

\subsection{Hermitian Lie algebras} Let $\fg$ be a Lie algebra endowed with an inner product
$g$. A connection $\nabla$ on $\fg$ is called metric if $\nabla
g=0$, that is, $\nabla_x$ is a skew-symmetric endomorphism of
$\fg$ for any $x\in\fg$. When $\fg$ is solvable and the connection
$\nabla$ is metric and flat, we have the following consequence:

\begin{lem}\label{flat-metric}
Let $\fg$ be a solvable Lie algebra equipped with an inner product
$g$ and a flat metric connection $\nabla$. Then $\{\nabla _x :
x\in \fg\}$ is a commutative family of skew-symmetric
endomorphisms of $\fg$ and $\nabla _v=0$ for every $v\in \fg '$.
\end{lem}

\begin{proof}
Since $\nabla$ is flat it defines a representation of $\fg$ on
itself, hence the image of $\nabla$ is a solvable Lie subalgebra
of $\fg \fl(\fg)$. Moreover, since $\nabla g=0$, this solvable Lie
subalgebra is contained in $\fs\fo(\fg)$, therefore it is abelian.
Now $\nabla _{[x,y]}= [\nabla _x , \nabla _y ]=0$, and the lemma
follows.
\end{proof}

\medskip

The Levi-Civita connection $\nabla^g$ associated to $g$ is the
only torsion-free metric connection on $\fg$, and it is given by
\begin{equation}\label{LC}
2g\left(\nabla^g_x y,z \right)= g\left([x, y],z \right)- g\left([
y,z],x \right)+ g\left([z,x],y \right), \qquad x,y,z \in \fg.
\end{equation}

If $J$ is a complex structure on $\fg$ compatible with $g$, that
is, $g(x,y) = g (Jx,Jy) $ for all $x,y \in \fg$, then $(J,g)$ is
called a Hermitian structure on $\fg$ and the triple $(\fg, J, g)$
is a Hermitian Lie algebra. The associated K\"ahler form $\omega$
is defined by $\omega (x,y)= g(Jx,y)$.  If $G$ is a Lie group with
Lie algebra $\mathfrak g$, by left-translating $J$ and the inner
product $g$ we give to $G$ a left invariant Hermitian structure. A
Hermitian Lie algebra $(\mathfrak g, J, g)$  is called
\textit{K\"ahler} if $\nabla^g$ is a complex connection or,
equivalently, $d\omega=0$, where \begin{equation}\label{d-omega}
d\omega(x,y,z)=-\omega([x,y],z)-\omega([y,z],x)-\omega([z,x],y),
\qquad x,y,z \in \fg.\end{equation}

\

\section{Abelian double products and affine Lie algebras}

In this section we consider a particular class of Lie algebras
equipped with abelian complex structures, the so called abelian
double products, and as a consequence of the main result of this
section (Theorem \ref{aff}) we obtain that $\fg'_J$ belongs to
this class, whenever $\fg$ is a Lie algebra and $J$ an abelian
complex structure.

\medskip

A \textit{ complex product structure} on $\fg $ is a complex
structure $J$ together with a decomposition  $\fg = \fu\oplus
J\fu$, where both, $\fu $ and $J\fu $, are Lie subalgebras of
$\fg$.
If $G$ is a Lie group with Lie algebra $\fg$, the
subalgebras $\fu$ and $J\fu$ give rise to involutive distributions
$TG^+$ and $TG^-$, respectively,  such that $TG= TG^+ \oplus TG^-$
and $J (TG^+)=TG^-$, where $J$ is the induced left invariant
complex structure on $G$. We obtain in this way a left invariant
complex product structure on $G$ (see \cite{AS}). 

It was proved in \cite{AD} that the Lie subalgebras $\fu $ and
$J\fu $ are abelian if and only if the complex structure $J$ is
abelian. In this case we will say that $\fg$ is an {\em abelian
double product}.

\

\begin{example} \label{ab_doub_aff}
{\rm  The Lie algebra of the affine motion group of $\cc$, denoted
$\aff (\cc)$, is an abelian double product with respect to two
non-equivalent abelian complex structures. Indeed, $\aff (\cc)$
has a basis $\{e_1,e_2,e_3,e_4\}$ with the following Lie bracket
\begin{equation}\label{corch_aff(C)}
[e_1,e_3]=e_3,\quad [e_1,e_4]=e_4, \quad [e_2,e_3]=e_4,\quad
[e_2,e_4]=-e_3. \end{equation} Any abelian complex structure on
$\aff (\cc)$ is equivalent to one and only one of the following
(see \cite[Proposition 2.6]{ABD1}):
\[ J_1e_1=-e_2,\qquad J_1e_3=e_4, \qquad \text{ or }\qquad J_2e_1=e_3,\qquad J_2e_2=e_4. \]
It follows  that $\aff (\cc )$ is an abelian double product with
respect to both $J_1$ and $J_2$, by setting $\fu _1 =\text{span}
\{ e_1+e_3, e_2+e_4\}$ and  $\fu_2 =\text{span} \{ e_1, e_2\}$,
respectively. }\end{example}

\

We introduce next a family of Lie algebras which exhausts the
class of abelian double products. They are obtained by considering
a finite dimensional real vector space $\mathcal A$ with two
structures of commutative associative algebra, $(\mathcal
A,\cdot)$ and $(\mathcal A,\ast)$, such that both products satisfy
the compatibility conditions
\begin{equation}\label{comp} a\ast (b\cdot
c)= b \ast (a\cdot c), \quad \quad a\cdot (b\ast c)= b\cdot(a \ast
c),\end{equation} for every $a, b, c \in \mathcal A$. Then,
$\mathcal A \oplus \mathcal A$ with the bracket:
\begin{equation}\label{mono} [(a,a'),(b,b')]=(- (a \ast b'-b\ast a'),
a\cdot b'-b\cdot a'), \hspace{1cm} a,b,a',b' \in \mathcal A,
\end{equation}
and the endomorphism $J$ defined by
\begin{equation}\label{standard}
J(a,a')=(-a',a),    \hspace{1cm} a,a'  \in \mathcal A ,
\end{equation}
is an abelian double product  that will be denoted $(\mathcal A ,
\cdot)\bowtie (\mathcal A , \ast)$ (see \cite{AS}). $J$ will be
called the {\em standard} complex structure on $(\mathcal A ,
\cdot)\bowtie (\mathcal A , \ast)$. Setting $\fu=\mathcal A \oplus
\{0\}$, it turns out that $J \fu= \{0\} \oplus \mathcal A$ and
both, $\fu $ and $J\fu $, are abelian Lie subalgebras of
$(\mathcal A , \cdot)\bowtie (\mathcal A , \ast)$. \footnote{ We
point out that the double product $(\mathcal A , \cdot)\bowtie
(\mathcal A , \ast)$ is a nilpotent Lie algebra if and only if
both $(\mathcal A , \cdot)$ and $(\mathcal A , \ast)$ are
nilpotent commutative  associative algebras \cite{A6}. }

Consider an abelian double product $(\mathcal A , \cdot ) \bowtie
(\mathcal A , \ast )$. In the special case when $\ast$ is the
trivial product in $\mathcal A$, the corresponding  Lie algebra is
called an {\em affine Lie algebra} and is denoted by
$\aff(\mathcal A)$. Two particular cases occur when $(\mathcal A,
\cdot)=\rr$ or $\cc$, obtaining in this way the Lie algebra of the
group of affine motions of either ${\rr}$ or ${\cc}$. The family
of Lie algebras $\aff(\mathcal A)$ where $\mathcal A$ is an
arbitrary associative algebra, not necessarily commutative, was
considered in \cite{ba-do}.

Condition \eqref{comp} is also satisfied in the particular case
when $a \ast b = a \cdot b$ for all $a, b \in \mathcal A$. In this
situation, we obtain:

\begin{lem}\label{iso}
Given a commutative associative algebra $(\mathcal A,\cdot)$,
there is a holomorphic Lie algebra isomorphism $\beta: (\mathcal
A, \cdot)\bowtie (\mathcal A, \cdot)\to \aff(\mathcal A)$ with
respect to the standard complex structures, given by
$\beta(a,b)=(a+b,-a+b)$, $a,b \in \mathcal A$.
\end{lem}

\smallskip

The  next  result,  which is a consequence of \cite[Proposition
6.1]{AD} and  \cite[Theorem 4.1]{AS}, shows that any  abelian
double product is obtained as in \eqref{mono} with the standard
abelian complex structure defined in \eqref{standard}.

\begin{prop} \label{assoc-comm}
Let $\fg= \fu \oplus J\fu$ be an abelian double product. Then the
Lie bracket in $\fg$ induces two structures of commutative
associative algebras in $\fu$, $(\fu , \cdot ), \; (\fu , \ast)$,
such that $\fg= (\fu , \cdot ) \bowtie (\fu , \ast )$.
\end{prop}

\begin{proof}
The proposition follows by considering in $\fu$ the products
\begin{equation}
x\cdot y= J\left([Jx,y]_{J\fu}\right), \qquad \quad \, x \ast y =
[Jx,y]_{\fu}, \quad x,y \in \fu.\end{equation}
\end{proof}

As a particular case of Proposition \ref{assoc-comm}, we state the
following result, which will be used below.

\begin{corol}\label{coro}
Let  $\fg$ an abelian double product of the form $\fg=\fg'\oplus
J\fg'$. Then:
\begin{enumerate}
\item[(i)] The Lie bracket in $\fg$ induces a structure of commutative associative algebra on $\fg'$ given by
$x \ast y =[Jx,y]$;
\item[(ii)] If $\mathcal A$ denotes the associative algebra $(\fg',\ast)$ in {\rm (i)}, then $\mathcal A^2=\mathcal A$
and $\fg$ is holomorphically isomorphic to $\aff(\mathcal A)$ with
its standard complex structure.
\end{enumerate}
\end{corol}

\begin{proof}
(i) follows from Proposition \ref{assoc-comm}, noting that in this
case $x\cdot y=0$ and $x \ast y =[Jx,y]$ for any $x,y\in \fg'$.

(ii) From Lemma \ref{cap}  we have that both, $\fg'$ and $J\fg'$
are  abelian, and therefore,  $\fg'=[\fg',J\fg']$. This implies
that $\mathcal A^2=\mathcal A$. Setting
\[ \phi:\aff(\mathcal A)\to \fg, \quad \phi(x,y)=y-Jx, \]
we obtain that $\phi$ is a holomorphic isomorphism, where
$\aff(\mathcal A)$ is equipped with its standard complex
structure.
\end{proof}

\

We show next that there is a large family of Lie algebras with
abelian complex structure which are not abelian double products.
Let $\fg=\fa \oplus \fv$ where $\fa=\text{span}\{ f_1,f_2\}$  and
$\fv$ is a  $2n$-dimensional real vector space. We fix an
endomorphism $J$   of $\fg$ such that $J^2=-I, \; Jf_1=f_2\, $ and
$\, \fv$ is $J$-stable. Given a linear isomorphism $T$ of $\fv$
commuting with $J|_{\fv}$, we define a Lie bracket on $\fg$ such
that $\fa$ is an abelian subalgebra, $\fv$ is an abelian ideal and
the bracket between elements in $\fa$ and $\fv$ is given by:
\[
[f_1,v]=TJ(v), \hspace{2cm }[f_2,v]=T(v), \qquad \text{ for every
} v\in \fv .
\]
It turns out that $J$ is an abelian complex structure on $\fg$
(see \cite[Example 5.2]{ba-do}).

\smallskip

\begin{prop}
The Lie algebra $\fg$ is not an abelian double product, unless
$n=1$ and $\fg =\aff (\cc )$ with the abelian complex structure
$J_1$ from Example \ref{ab_doub_aff}.
\end{prop}

\begin{proof} Assume that $\fg=\fa \oplus \fv$ as above is an abelian double product, $\fg =\fu \oplus J\fu$, and let
$\{ x_1, \dots , x_{n+1}\}$ be a basis of $\fu$. For $i= 1, \dots
, n+1$, there exist $a_i,b_i \in \rr, \; e_i\in \fv$ such that:
\[ x_i=a_i f_1+b_i f_2+e_i. \]
Since $\fu$ is abelian we have
\[ 0=[x_i, x_k]= T\left( (b_i I+a_i J) e_k - (b_k I+a_k J) e_i \right), \qquad \text{ for every } i,k, \]
therefore, as $T$ is an isomorphism, it follows that
\begin{equation} \label{dimu=2}
(b_i I+a_i J) e_k = (b_k I+a_k J) e_i , \qquad \text{ for every }
i,k .
\end{equation}
We note that $a_i^2+b_i^2\neq 0$ for some $i$, otherwise we would
have $\fu \subset \fv$, hence $\fg \subset \fv$, a contradiction.
We may assume that $a_1^2+b_1^2\neq 0$ and, by applying  $(b_1
I+a_1 J)^{-1}$ on both sides of \eqref{dimu=2} with $i=1$, we
obtain:
\[
e_k= (b_1 I+a_1 J)^{-1}(b_k I+a_k J) e_1 , \qquad \text{ for every
} k.
\]
Setting $\fs= \text{span}\{ f_1, f_2, e_1, Je_1\}$, the above
equation implies that $x_k \in \fs $ for every $k$, hence, $\fu
\subset \fs$. Since $\fs$ is $J$-stable, it follows that $\fu
\oplus J\fu \subset \fs$, therefore, $\fg = \fs$. From \cite[Lemma
2.8]{ABD1} we conclude that $\fs \cong \aff (\cc)$ with the
abelian complex structure $J_1$ from Example \ref{ab_doub_aff}.
\end{proof}

\medskip

In the following result we give a characterization of affine Lie
algebras among the Lie algebras carrying an abelian complex
structure. Any affine Lie algebra $\fg=\aff(\mathcal A)$ can be
written as $\fg=\fa\oplus J\fa$ with $\fa$ an abelian ideal and
$J\fa$ an abelian subalgebra and, moreover, $\fg'\cap
J\fg'=\{0\}$. We show next that these conditions are also
sufficient.

\begin{teo} \label{aff}
Let $\fg$ be a solvable Lie algebra with an abelian complex structure $J$
such that $\fg$ admits a vector space decomposition $\fg=\fu +
J\fu$. Then:
\begin{enumerate}
\item[(i)] if $\fu$ is an abelian subalgebra of $\fg$ then $\fg = \fa \oplus J\fa$
 is an abelian double product with $\fa \subset \fu$;
\item[(ii)] if $\fu$ is an abelian ideal of $\fg$ and, moreover, $\fg'\cap J\fg'=\{0\}$, then $(\fg, J)$ is
holomorphically isomorphic to $\aff(\mathcal A)$ for some
commutative associative algebra $(\mathcal A, \cdot)$.
\end{enumerate}
\end{teo}

\begin{proof}
We may assume, in either case, that $\fu\cap J\fu=\fz$, where
$\fz$ denotes the center of $\fg$. Indeed, $\fu+\fz$ is an abelian
subalgebra (respectively, abelian ideal) such that $\fg=(\fu+\fz)
+ J(\fu+\fz)$ and $(\fu+\fz) \cap J(\fu+\fz)=\fz$. To prove the
last equality, let $x+z = J(y+w)$ where $x,y\in \fu $ and $z,w\in
\fz$. Then $[x+z, u]=0$ for all $u \in  \fu$ and
\[ [x+z, Ju]= - [J(x+z), u]= [y +w, u]= 0, \quad u \in  \fu.
\] Since $\fg=\fu + J\fu$ and $\fz$ is $J$-invariant, the
assertion follows.

(i) Let $\fu=\fh+\fz$ and $\fz=\fl \oplus J\fl$. Setting $\fa =
\fh \oplus \fl$, one obtains $\fg=\fa \oplus J\fa$, with $\fa
\subset \fu$, hence $\fa$ is an abelian subalgebra, and (i)
follows.

\smallskip

(ii) Since in this case $\fu$ is an abelian ideal, one has that
$\fg'\subset \fu$. Let $\fh \subset \fg'$ and $\fk \subset \fu$ be
subspaces  such that \[ \fg'=\fh \oplus (\fg'\cap \fz), \qquad \fu
=\fk \oplus (\fg'+ \fz). \]

Using the condition $\fg'\cap J\fg'=\{0\}$, we may decompose $\fz$
as
\[\fz = (\fg'\cap \fz) \oplus\fl \oplus J(\fg'\cap \fz) \oplus
J\fl.\]

Set now $\fv=\fk \oplus \fh \oplus (\fg'\cap \fz) \oplus\fl\subset
\fu$. It can be verified that $\fv\cap J \fv=\{0\}$ and $\dim
\fv=\frac12 \dim \fg$, hence $\fg=\fv\oplus J \fv$. Since $\fg'
\subset \fv \subset \fu $, it follows that $\fv$ is an abelian
ideal. Arguing as in the proof of Corollary \ref{coro}, we obtain
that $\fg=\fv\oplus J \fv$ is holomorphically isomorphic to
$\aff(\mathcal A)$ where $\mathcal A$ is the commutative
associative algebra $(\fv,*)$ where $x* y=[Jx,y], \, x,y\in \fv$.
\end{proof}

\begin{corol}
Let $\fg$ be a solvable Lie algebra with an abelian complex
structure $J$. Then:
\begin{enumerate}
\item $\fg'_J$ is an abelian double product and if $\fg'\cap
J\fg'=\{0\}$, then $(\fg'_J, J)$ is holomorphically isomorphic to
$\aff(\mathcal A)$ for some commutative associative algebra
$(\mathcal A, \cdot)$;
\item if $\fg=\fg' + J\fg'$, then $\fg= \fu \oplus J\fu$ is an abelian
double product for some subalgebra $\fu \subset \fg '$.
\end{enumerate}
\end{corol}

\

\section{Hermitian Lie algebras with abelian complex structures}

In this section we study natural questions concerning
Hermitian Lie algebras with an abelian complex structure, namely,
the existence of K\"ahler structures and the
flatness of the first canonical Hermitian connection.

\subsection{K\"ahler Lie algebras with abelian complex structures.}
We show next that such a Lie algebra is affine of a very
restrictive type. More precisely, its associated simply connected
Lie group is a direct product of $2$-dimensional hyperbolic spaces
and a flat factor.

\begin{teo}\label{kahler-abelian}
Let $(\mathfrak g, J, g)$ be a K\"ahler Lie algebra with $J$ an
abelian complex structure. Then $\fg$ is isomorphic to
\[ \aff(\rr)\times \cdots \times \aff(\rr)\times \rr^{2s},\] and this decomposition is orthogonal and $J$-stable.
\end{teo}

\begin{proof}
Combining the K\"ahler condition $d\omega=0$ with \eqref{d-omega}
we obtain
\[
0=d\omega(Jx,Jy,Jz)=-\omega([Jx,Jy],Jz)-\omega([Jy,Jz],Jx)-\omega([Jz,Jx],Jy).
\]
The fact that $J$ is abelian gives
\begin{equation}\label{kahler}
g([x,y],z)+g([y,z],x)+g([z,x],y)=0 \end{equation} for any
$x,y,z,\in\fg$, and from this expression the following statements
are easily verified:
\begin{enumerate}
\item[(i)] $\fz\subseteq (\fg'_J)^{\perp}$.
\item[(ii)] $(\fg')^{\perp}$ is abelian.
\item[(iii)] $\ad_z|_{\fg'}$ is symmetric for all $z\in\fg$.
\end{enumerate}

We prove next that we have an equality in (i). Let
$x\in(\fg'_J)^{\perp}$, then $J\ad_x=-\ad_xJ$. Indeed, we compute
\begin{align*}
g(J[x,y],z) & = -g([x,y],Jz)\\
            & = g([y,Jz],x)+g([Jz,x],y) \\
            & = g([Jz,x],y)=-g([z,Jx],y)\\
            & = g([Jx,y],z)+g([y,z],Jx)\\
            & = g([Jx,y],z) \\
            & = -g([x,Jy],z)
\end{align*}
and the claim follows. Since $(\fg'_J)^{\perp}\subset
(\fg')^{\perp}$, it follows from (ii) that $(\fg'_J)^{\perp}$ is
abelian. Furthermore, it is   $J$-stable and therefore  we have
\[
(\ad_x)^2J=-\ad_x\ad_{Jx}=-\ad_{Jx}\ad_{x}=\ad_x J \ad_x =
-(\ad_x)^2 J
\]
which implies $(\ad_x)^2=0$. Since $\ad_x$ is also symmetric
(which follows easily from \eqref{kahler}), we obtain that
$\ad_x=0$, that is, $x\in\fz$ and therefore
\begin{equation}\label{decom}
\fz = (\fg'_J)^{\perp} \text{\qquad and \qquad }
\fg=(\fg'_J)\oplus \fz.
\end{equation}
It follows from Lemma \ref{cap}(v) that $\fg'\cap J\fg'=\{0\}$.
Hence,
\begin{equation}\label{direct1}
\fg=(\fg'\oplus J\fg')\oplus \fz.
\end{equation}

Now, from Corollary \ref{coro} applied to $\fg'\oplus J\fg'$, we
obtain that there is a structure of commutative associative
algebra on $\fg'$ given by
\[ x\ast y= [Jx,y], \qquad  \text{for } x,y \in \fg'.\]
We denote $\mathcal A=(\fg',\ast)$ and $\ell_x:\fg'\to\fg'$ the
multiplication by $x\in\fg'$. Using that $J$ is abelian and
condition (iii) above, we obtain that $\ell_x$ is symmetric with
respect to $g|_{\mathcal A\times \mathcal A}$ for any $x\in
\mathcal A$. If $x_0\in \mathcal A$ is a nilpotent element, then
$\ell_{x_0}$ is nilpotent. Since $\ell_{x_0}$ is also symmetric,
we have that $\ell_{x_0}=0$, that is, $x_0\in\fz$. It follows from
\eqref{direct1} that $x_0=0$. Thus, $\mathcal A$ has no nilpotent
elements, and therefore it is semisimple. From structure theory of
associative commutative algebras, we have that $\mathcal A$ is a
direct sum of copies of $\rr$ and $\cc$ with their canonical
product. Moreover, $\cc$ cannot occur in the decomposition since
$\ell_x$ has real eigenvalues for any $x\in\fg'$. We conclude that
\[ \mathcal A=\rr e_1 \oplus \cdots \oplus \rr e_n, \qquad \text{with} \qquad e_i^2=e_i,\, e_ie_j=0, \, i\neq j.\]
This gives a $J$-stable splitting
\[ \fg'\oplus J\fg'=\aff(\rr e_1)\times \cdots \times  \aff(\rr e_n).\]
Moreover, using that $e_i^2=e_i,\, e_ie_j=0, \, i\neq j$ for
$i,j=1,\ldots, n$, we obtain that the splitting above is
orthogonal.
\end{proof}

\smallskip

\begin{corol}\label{kahler-group}
Let $G$ be a simply connected Lie group equipped with a
left invariant K\"ahler structure $(J,g)$ such that $J$ is
abelian. If the commutator subgroup is $n$-dimensional  and the
center is $2s$-dimensional, then
\[ G=H^2(-c_1)\times \cdots \times H^2(-c_n) \times \RR^{2s}, \]
where $c_i>0,\, i=1,\ldots,n$, and $H^2(-c_i)$ denotes the
$2$-dimensional hyperbolic space of constant curvature $-c_i$.
\end{corol}

\begin{proof}
Let $e_i\in \fg'$ be given as in the proof of Theorem
\ref{kahler-abelian}, and let $r_i=\Vert e_i \Vert = \Vert Je_i \Vert$. Then the
simply connected Lie group with Lie algebra $\aff(\rr e_i)$ is
isometric and bilohomorphic to $H^2(-c_i)$, where
$c_i=\frac{1}{r_i^2}$.
\end{proof}

\smallskip

If $G$ is a simply connected Lie group and $\Gamma\subset G$ is a
discrete subgroup then any left invariant Hermitian structure on
$G$ gives rise to a unique Hermitian structure on $\Gamma
\backslash G$ such that the projection $G\to \Gamma \backslash G$
is a holomorphic local isometry. This structure on $\Gamma
\backslash G$ will also be called left invariant.

\begin{corol} Let $M=\Gamma \backslash G$ be a compact quotient with a left invariant K\"ahler structure $(J,g)$ such that $J$ is abelian. Then $M$ is diffeomorphic to a torus.
\end{corol}

\begin{proof}
As $G$ admits a compact quotient, then $G$ is a unimodular Lie
group. Using the characterization given in Theorem
\ref{kahler-abelian}, together with the fact that $\aff(\RR)$ is
not unimodular, the corollary follows.
\end{proof}

\smallskip

\begin{rem}
{\rm We point out that the previous corollary holds without the assumption of the left invariance of the metric $g$. Indeed, following \cite{Be} (see also \cite{u}) any K\"ahler metric with respect to a left invariant complex structure gives rise to a left invariant K\"ahler metric with respect to the same complex structure.}
\end{rem}

\begin{rem}
{\rm In \cite{BB} the pseudo-Riemannian geometry of abelian para-K\"ahler Lie algebras is studied. In particular, the authors provide conditions to ensure the flatness or Ricci-flatness of the para-K\"ahler (neutral) metric.}
\end{rem}

\medskip

\subsection{The first canonical Hermitian connection}

Given a Hermitian Lie algebra $(\fg,J,g)$, consider the connection
$\nabla^1$ on $\fg$ defined by
\[
g\left( \nabla^1_x y , z \right) =  g\left( \nabla^g _x y,z
\right) + \frac 14 \left( d\omega (x,Jy,z) + d\omega (x,y,Jz)
\right) ,
\]
where $\omega$ is the K\"ahler form. This connection satisfies
\[
\nabla^1 g =0, \quad \nabla^1 J= 0, \quad T^1  \; \text{ is of type } (1,1).
\]

The connection $\nabla^1$ is known as the first canonical
Hermitian connection associated to the Hermitian Lie algebra
$(\fg,J,g)$ (see \cite{Li,Ga}). It is proved in \cite[p. 21]{A}
that another expression for $\nabla^1$ in terms of the Levi-Civita
connection $\nabla^g$ is given by
\begin{equation}\label{1-nabla}
\nabla^1_{x}y := {\nabla^g}_x y  + \frac{1}{2} \left(
{\nabla^g}_{x}J\right)Jy = \frac{1}{2}({\nabla^g}_x y -J
{\nabla^g}_x Jy ),
\end{equation}
for $x,y \in \fg$. More generally, if ${\nabla}$ is anyconnection on $\mathfrak g$, define
\begin{equation}\label{nabla1}
\overline{\nabla}_{x}y := {\nabla}_x y  + \frac{1}{2} \left(
{\nabla}_{x}J\right)Jy = \frac{1}{2}({\nabla}_x y -J {\nabla}_x Jy
),
\end{equation}
for $x,y \in \mathfrak g$. It is easy to see that
$\overline{\nabla} J=0$ and, furthermore, if $\nabla$ is
torsion-free, then $\overline{T}(x,y)=\overline{T}(Jx,Jy)$, i.e.
$\overline{T}$ is of type $(1,1)$ with respect to $J$.

\medskip

We prove in the next result that when $\nabla$ is torsion-free and
$\overline{\nabla}$ coincides with the $(-)$-connection, then the
complex structure $J$ is abelian, and therefore, the Lie algebra
is $2$-step solvable.

\begin{lem}\label{first=0}
Let $\nabla$ be a torsion-free connection and $J$ a complex
structure on $\mathfrak g$. Assume that $\overline{\nabla}=0$,
that is, $\nabla_x J =-J \nabla _x$ for every $x\in \mathfrak g$.
Then $J$ is abelian.
\end{lem}

\begin{proof}
If $\overline{\nabla}=0$, then $\overline{T}(x,y)=-[x,y]$, where
$\overline{T}$ denotes the torsion of $\overline{\nabla}$. Since
$\overline{T}$ is of type $(1,1)$, it follows that $J$ is abelian.
\end{proof}

\medskip

If in Lemma \ref{first=0} the connection $\nabla$ is the
Levi-Civita connection of a Hermitian metric on $(\fg, J)$, then a
much stronger restriction occurs, namely, $\fg$ is abelian. This
is the content of the next theorem.

\begin{teo}\label{nabla-first} Let $(\fg, J, g)$ be
a Hermitian Lie algebra such that its associated first canonical
connection $\nabla^1$ satisfies $\nabla^1_xy=0$ for every $x,y\in
\fg$, that is, $\nabla^1$ coincides with the $(-)$-connection.
Then $\fg$ is abelian.
\end{teo}

\begin{proof}
Let $\nabla$ be the Levi-Civita connection of $g$. From
$\nabla^1\equiv 0$ and equation \eqref{nabla1}, it follows that
$\nabla_x J =-J \nabla _x$ for every $x\in \mathfrak g$, therefore
\begin{equation}\label{paso1}
g\left(\nabla _xJ y,z \right) =-g\left(J\nabla _x y,z \right)
=g\left(\nabla _x y,Jz \right),\end{equation} for every $x,y,z \in
\fg$. Moreover, Lemma \ref{first=0} implies that
 $J$ is abelian.
Using equation \eqref{LC} for the first and third terms in \eqref{paso1} and the fact that $J$ is abelian, we get:
\begin{equation}\label{cyclic}
g\left([x, Jy],z \right)- g\left([x, y],Jz \right)+2 g\left([
y,Jz],x \right)+g\left([z,x],Jy \right)-g\left([Jz,x],y \right)=0.
\end{equation}
Computing the sum of  \eqref{cyclic} over cyclic permutations of
$x,y,z$ we obtain the following equation:
\begin{equation}\label{vip}
 g\left([x, Jy],z \right)+g\left([ y,Jz],x \right)+g\left([ z,Jx],y \right)=0,
\end{equation}
for every $x,y,z \in \fg$. The following conditions are
consequences of \eqref{vip}:
\begin{enumerate}
\item[(i)] $\fz\subseteq (\fg'_J)^{\perp}$.
\item[(ii)] $(\fg'_J)^{\perp}$ is abelian.
\item[(iii)] $\ad_x$ is skew-symmetric for all $x\in(\fg'_J)^{\perp}$.
\end{enumerate}

We prove next that we have an equality in (i). Let
$x\in(\fg'_J)^{\perp}$, then $J\ad_x=-\ad_xJ$. Indeed, since $J$
is abelian, we have $\ad_x J=-\ad_{Jx}$, and therefore it follows
from (iii) that both $\ad_x$ and $\ad_x J$ are skew-symmetric. By
taking the adjoint of $(\ad_x J)$, the claim easily follows.
Since $(\fg'_J)^{\perp}$ is abelian and $J$-stable (see (ii)), we
have
\[
(\ad_x)^2J=-\ad_x\ad_{Jx}=-\ad_{Jx}\ad_{x}=\ad_x J \ad_x =
-(\ad_x)^2 J
\]
which implies $(\ad_x)^2=0$. From (iii) above we obtain $\ad_x=0$,
that is, $x\in\fz$, and therefore
\begin{equation}\label{decom1}
\fz = (\fg'_J)^{\perp} \text{\qquad and \qquad }
\fg=(\fg'_J)\oplus \fz. \end{equation} It follows from Lemma
\ref{cap}(v) that $\fg'\cap J\fg'=\{0\}$. Hence,
\begin{equation}\label{direct}
\fg=(\fg'\oplus J\fg')\oplus \fz.
\end{equation}

Now, from Corollary \ref{coro} applied to $\fg'\oplus J\fg'$, we
obtain that $\fg=\aff(\mathcal A)\oplus \fz$, with $\mathcal
A^2=\mathcal A$. Furthermore, equation \eqref{vip} can be read as
\begin{equation}\label{vip2}
g(x\ast y,z)+g(y\ast z,x)+ g(z\ast x,y)=0, \qquad x,y,z\in \fg'.
\end{equation}
Let us suppose that $\mathcal A\neq \{0\}$, so that $\mathcal A$
is not nilpotent. There exists an idempotent element $e\in
\mathcal A$, $e\neq 0$, $e^2=e$. Setting $x=y=z=e$ in
\eqref{vip2}, one obtains $g(e,e)=0$, a contradiction. Therefore
$\mathcal A=\{0\}$ and this means that $\fg$ is abelian.
\end{proof}

\smallskip

\begin{rem}
{\rm A similar result for the Chern connection does not hold. See
\cite{DV} for results concerning the flatness of the Chern
connection on nilmanifolds.}
\end{rem}

\begin{rem}
{\rm In contrast with Theorem \ref{nabla-first}, we point out that every abelian double product can be endowed with a 
torsion-free connection $\nabla$ such that $\overline{\nabla}=0$.  }
\end{rem}

\medskip

If $G$ is a simply connected Lie group with Lie algebra $\fg$ and
$\Gamma \subset G$ is any discrete subgroup of $G$ then the
$(-)$-connection on $\fg$ induces a unique connection on $\Gamma
\backslash G$ such that the parallel vector fields are
$\pi$-related with the left invariant vector fields on $G$, where
$\pi:G\to\Gamma \backslash G$ is the projection. This induced
connection on $\Gamma \backslash G$, which will be denoted
$\nabla^0$, is complete, has trivial holonomy, its torsion is
parallel and $\pi$ is affine.

As a consequence of Theorem \ref{nabla-first}, we obtain

\begin{corol} Let $M=\Gamma \backslash G$ be a compact quotient of a
simply connected Lie group $G$ by a discrete subgroup $\Gamma$. If
$(J,g)$ is a left invariant Hermitian structure on $M$ such that
its first canonical connection $\nabla^1$ coincides with the
connection $\nabla^0$, then $M$ is diffeomorphic to a torus.
\end{corol}

\medskip

Motivated by Theorem \ref{nabla-first}, we study next the more
general case of Hermitian structures on Lie algebras whose
associated first canonical connection is flat, when the complex
structure is abelian. The following lemma will be useful to prove some of the next results.

\begin{lem}\label{cap1}
Let $(\fg, J, g)$ be a Hermitian Lie algebra with $J$ abelian.  If the associated
first canonical connection $\nabla^1$ is flat, then $\fz \cap \fg '=\{0\}$.
\end{lem}

\begin{proof}
It follows from equations \eqref{1-nabla} and \eqref{LC} and the fact that $J$ is abelian that:
\begin{equation}\label{eq-nabla1st}
g \left( \nabla^1_{x}y ,z \right)=  \frac{1}{4}\left(
g \left( [x,y],z\right)  +g \left( [z,x],y \right)
+ g \left( [x,Jy],Jz\right)+g \left( [Jz,x],Jy \right)- 2 g \left( [y,z],x\right)
 \right),
\end{equation}
for every $x,y,z \in \fg$. If $x\in \fz \cap \fg '$, using Lemma \ref{flat-metric}, equation \eqref{eq-nabla1st} becomes
\[ 0= g \left( \nabla^1_{x}y ,z \right)= - \frac12 g \left( [y,z],x\right),\]
therefore, $x\in (\fg ')^{\perp}$ which implies $x=0$.
\end{proof}

\medskip

As a straightforward consequence of Lemma \ref{cap1}, 
we show next that a nilpotent Lie algebra with an abelian complex
structure admits no Hermitian metric with flat first canonical
connection. In Theorem \ref{abeliana1} below we extend this result to an arbitrary Lie algebra carrying an abelian complex structure. 

\begin{prop}\label{2-pasos}
Let $(\fg, J, g)$ be a Hermitian Lie algebra with $\fg$ nilpotent
and $J$ abelian. If the associated first canonical connection
$\nabla^1$ is flat, then $\fg$ is abelian.
\end{prop}
\begin{proof}
Consider the descending central series of $\fg$  defined by
$\fg ^0 = \fg ,  \; \fg ^i=[\fg , \fg ^{i-1}], \; i\geq 1$. The Lie algebra $\fg$ is  $k$-step nilpotent if  $\fg^k=\{0\}$ and $\fg^{k-1}\neq \{0\}$.

Let us suppose that
 $\fg$ is  $k$-step nilpotent with $k\geq 2$. Then $\fg^{k-1}\subset \fz \cap \fg '$ and it follows from Lemma \ref{cap1} that $\fg^{k-1}=\{0\}$, a contradiction. Therefore, $k=1$ and $\fg$ is abelian.
\end{proof}

\medskip

\begin{corol}Let $N= \Gamma\backslash G$ be a nilmanifold  with a left invariant Hermitian structure
$(J,g)$ such that $J$ is abelian. If the associated first canonical connection $\nabla^1$ is flat, then $N$
is diffeomorphic to a torus.
\end{corol}

\medskip

The next result shows that the nilpotency assumption can be dropped in  Proposition \ref{2-pasos}. In other words, there are no non-abelian Hermitian Lie algebras with abelian complex structure and flat first canonical connection.

\medskip

\begin{teo}\label{abeliana1}
Let $(\fg, J, g)$ be a Hermitian Lie algebra with $J$ abelian. If the associated 
first canonical connection $\nabla^1$ is flat, then $\fg$ is abelian.
\end{teo}

\smallskip

The proof of this theorem will follow from Lemmas \ref{proper} and \ref{previous}, which are proved next.

\smallskip

\begin{lem}\label{proper}
Let $(\fg, J, g)$ be a Hermitian Lie algebra with $J$ abelian. If the associated 
first canonical connection $\nabla^1$ is flat, then both, $\fg'_J$ and $(\fg'_J)^{\perp}$ are  abelian. Moreover, if
$v,w\in (\fg'_J)^{\perp}$ then $\nabla^1_vw=0$. 
\end{lem}

\begin{proof} 
For any $x,y,z \in \fg'$, using \eqref{eq-nabla1st} together
with Lemma \ref{flat-metric} and the fact that $\fg'$ is abelian
we obtain:
\[
0=g\left(\nabla^1_xy, Jz  \right)= \frac14\left( g([Jz,x],y)-g([x,Jy],z)-2g([y,Jz],x) \right),
\]
or equivalently:
\begin{equation}\label{star}  2g([y,Jz],x)=-g([z,Jx],y)-g([x,Jy],z).
\end{equation}
Using \eqref{star} we compute
\[ 2g([y,Jz],x) - 2g([z,Jx],y)= -g([z,Jx],y)+g([y,Jz],x), \]
which implies $g([y,Jz],x) =g([z,Jx],y)$ for every $x,y,z \in
\fg'$. Therefore,  $\ad_{Jz }:\fg' \to \fg'$ is
symmetric for all $z\in \fg'$. In particular, using that
$\ad_{Jx}$ is symmetric in equation \eqref{star} we obtain
\[ 2g([y,Jz],x)=-2 g([x,Jy],z), \quad \text{equivalently,} \quad g([Jy,z],x)=-g([Jy,x],z), 
\]
that is, $\ad_{Jy}$ is skew-symmetric for all $y\in \fg'$.
Hence, $\ad_{Jy}|_{\fg'}=0$  for all $y\in \fg'$ which implies that $\fg'_J$ is abelian. 

Now taking $x\in \fg', \; y,z \in (\fg'_J)^{\perp}$ in \eqref{eq-nabla1st} together
with Lemma \ref{flat-metric} and the fact that $(\fg'_J)^{\perp}$ is $J$-stable we obtain:
\[
0=g\left(\nabla^1_xy, z  \right)= -\frac12 g([y,z],x),
\]
and therefore, $(\fg'_J)^{\perp}$ is abelian.

Since $(\fg'_J)^{\perp}$ is abelian and $J$-stable, it follows from \eqref{eq-nabla1st} that $\nabla^1_vw=0$ whenever 
$v,w\in (\fg'_J)^{\perp}$.
\end{proof}

\medskip

\begin{lem}\label{previous}
Let $(\fg, J, g)$ be a Hermitian Lie algebra with $J$ abelian. If the associated 
first canonical connection $\nabla^1$ is flat then $\fg'$ is $J$-stable and therefore $\fg'_J=\fg'$.
\end{lem}

\begin{proof}
Let $\fv$ denote the orthogonal complement of $\fg'$ in $\fg'_J$, so that $\fg'_J=\fg'\oplus \fv$. Taking $x\in\fg',\, y\in\fv,\, w\in (\fg'_J)^{\perp}$ and using \eqref{eq-nabla1st} together with Lemma \ref{flat-metric} we obtain
\[ 0=g\left(\nabla^1_xy,w \right) = g([Jw,x],Jy)-2g([y,w],x), \]
and therefore
\[ 2g([w,y],x)= -g([Jw,x],Jy). \]
In particular, if $x=[w,y]$, we have
\begin{equation}\label{norma} 2 \Vert [w,y]\Vert^2=-g([Jw,[w,y]],Jy)=-g([w,[Jw,y]], Jy)\end{equation}
using Jacobi identity and Lemma \ref{proper}. Replacing $w$ by $Jw$ in \eqref{norma}, we get
\[ 2 \Vert [Jw,y]\Vert^2=g([Jw,[w,y]], Jy)=g([w,[Jw,y]], Jy), \]
which together with \eqref{norma} implies $[w,y]=0$. Since $(\fg'_J)^{\perp}$ is abelian, it follows that $y\in \fz$, hence $\fv\subset \fz$. Now, $Jy\in\fg'_J \cap \fz$, hence we can write $Jy=x+v$, with $x\in\fg',\, v\in \fv$. It follows that $x\in \fz$, and using Lemma \ref{cap1} we obtain $x=0$, so that $\fv$ is $J$-invariant. Since $J$ is orthogonal, we have that $\fg'$ is also $J$-invariant, thus $\fv=0$ and $J\fg'=\fg'$.
\end{proof}

\medskip

\begin{proof}[Proof of Theorem \ref{abeliana1}]
It follows from Lemma \ref{previous} that $\fg'_J=\fg'$, and therefore there is an orthogonal decomposition $\fg=\fg'\oplus (\fg')^{\perp}$, where both subalgebras are 
$J$-stable and abelian (Lemma \ref{proper}). Taking $x,y\in \fg', \; z\in (\fg')^{\perp}$ in \eqref{eq-nabla1st} we 
obtain
\begin{equation}\label{otra}
g([z,x],y)-g([z,Jx],Jy)+2g([z,y],x)=0. 
\end{equation}
Changing $x$ by $Jx$ and $y$ by $Jy$ in the equation above, and adding both equations we get
\[
g([z,y],x)+ g([z,Jy],Jx) =0.
\]
Since this holds for every $x\in \fg'$ and $\fg'$ is $J$-invariant, we obtain:
\begin{equation}\label{j-menos}
J[z,y]=-[z,Jy], \quad z\in (\fg')^{\perp}, \, y\in\fg'. 
\end{equation}
Using this fact in \eqref{otra}, it follows that 
\[ g([z,x],y)+g([z,y],x)=0,\]
and therefore $\ad_z:\fg'\to\fg'$ is skew-symmetric for any $z\in (\fg')^{\perp}$.

Let us take now $z\in (\fg')^{\perp}$ and $x,y\in \fg'$ and compute
\begin{align*}
4g(\nabla^1_zx,y) & =  g([z,x],y)+g([y,z],x)+g([z,Jx], Jy)+g([Jy,z],Jx)-2g([x,y],z) \\
                 & =  2 g([z,x],y) + 2g([z,Jx], Jy) \\
                 & = 0,
\end{align*}
where we used that $\ad_z$ is skew-symmetric in the first equality and \eqref{j-menos} in the second equality. Combining
this with the fact that $\nabla^1_vw=0$ for $v,w\in (\fg')^{\perp}$ (Lemma \ref{proper}), we obtain that $\nabla^1_z=0$ 
for any $z\in (\fg')^{\perp}$. Since $\nabla^1_x=0$ for any $x\in \fg'$ (Lemma \ref{flat-metric}), it follows that 
$\nabla^1=0$, and therefore $\fg$ is abelian, according to Theorem \ref{nabla-first}.
\end{proof}

\medskip

\begin{rem} {\rm According to Theorem \ref{abeliana1}, there are no non-abelian Hermitian Lie algebras with abelian complex structure and flat first canonical connection. On the other hand, 
we note that there exist Hermitian Lie algebras with flat first canonical connection and non-abelian
complex structure. Indeed, in \cite{BDF} it was shown that any even-dimensional Lie algebra with a flat metric $g$
admits a compatible complex structure $J$ such that $(J,g)$ is K\"ahler, following results of Milnor in \cite{Mil}. As in this case we have that
$\nabla^1=\nabla^g$, it follows that $\nabla^1$ is flat. Moreover, $J$ is not abelian, unless the Lie algebra is
abelian (see Theorem \ref{kahler-abelian}). It would be interesting to find examples of flat first canonical connections in the non-K\"ahler case.} 
\end{rem}

\


\begin{thebibliography}{GG} \frenchspacing
\bibitem{A} { I. Agricola}, {\em The Srn\'\i\ lectures on non-integrable geometries with torsion}, Arch. Math.
(Brno) {\bf 42} (2006), 5--84.


\bibitem{A6} A. Andrada, {\it Complex product structures on 6-dimensional nilpotent Lie algebras},  Forum Math.
{\bf 20} (2008), 285--315.

\bibitem{ABD1} { A. Andrada, M. L. Barberis, I. Dotti}, {\it Classification of abelian complex structures on
$6$-dimensional  Lie algebras}, J. London Math. Soc. {\bf 83} (2011), 232--255.


\bibitem{ABD} A. Andrada, M. L. Barberis and I. Dotti, {\it Complex connections with trivial holonomy}, to appear in: Lie groups: Structure, Actions and Representations (Progress in Mathematics - Birkh\"auser).

\bibitem{AD} A. Andrada and I. Dotti, {\it Double products and
hypersymplectic structures on $\rr^{4n}$}, Commun. Math. Physics
{\bf 262} (2006), 1--16.

\bibitem{AS} A. Andrada and S. Salamon, {\it Complex product
structures on Lie algebras}, Forum Math. {\bf 17} (2005),
261--295.

\bibitem{BB} I. Bajo and S. Benayadi, {\it Abelian para-K\"ahler structures on Lie algebras}, Differ. Geom. Appl. {\bf 29} (2011), 160--~173.





\bibitem{ba-do} M. L. Barberis and I. Dotti, {\it Abelian complex structures on solvable Lie
algebras}, J. Lie Theory {\bf 14} (2004), 25--34.

\bibitem{BDF} M. L. Barberis, I. Dotti and A. Fino, {\it Hyper-K\"ahler quotients of solvable Lie groups},
J. Geom. Phys. {\bf 56} (2006), 691--711.

\bibitem{BDV} M. L. Barberis, I. Dotti and M. Verbitsky, {\it Canonical bundles of complex nilmanifolds,
with applications to hypercomplex geometry}, Math. Research Letters {\bf 16} (2) (2009), 331--347.

\bibitem{Be} F. A. Belgun, {\it On the metric structure of non-K\"ahler manifolds}, Math. Ann. {\bf 317} (2000), 1--40.


\bibitem{c} S. Console, {\it Dolbeault cohomology and deformations of
nilmanifolds}, Revista de la Uni\'{o}n Matem\'{a}tica Argentina,
{\bf 47} (2006), no. 1, 51--60 (Proceedings of II Encuentro de
Geometr\'{\i}a Diferencial, La Falda, Sierras de C\'{o}rdoba,
Argentina 2005.)


\bibitem{cfp} S. Console, A. Fino and Y. S. Poon, {\it Stability of abelian complex structures}, Internat. J. Math.
{\bf 17}  No. 4 (2006), 401--416.

\bibitem{cfu} L. A. Cordero, M. Fern\'{a}ndez and L. Ugarte, {\it  Abelian complex
structures on $6$-dimensional compact nilmanifolds}, Comment.
Math. Univ. Carolinae {\bf 43}(2)  (2002), 215--229 .

\bibitem{DV} A. Di Scala, L. Vezzoni, {\it Chern-flat and Ricci-flat invariant almost Hermitian  structures}, Ann. Glob. Anal. Geom. {\bf 40} (2011), 21--45.

\bibitem{cqg} I. Dotti and A. Fino, {\it Hyper-K\"ahler torsion structures invariant by nilpotent
Lie groups},  Classical Quantum Gravity {\bf 19} (2002), 551--562.


\bibitem{Ga}{ P. Gauduchon}, {\it Hermitian connections and Dirac operators},  Boll. Un. Mat. Ital. Ser. VII {\bf 2} (1997), 257--288.

\bibitem{gp} G. Grantcharov and Y. S. Poon, {\it Geometry of Hyper-K\"ahler connections with torsion},
Comm. Math. Phys. {\bf 213} (1) (2000), 19--37.

\bibitem{He} { S. Helgason}, {\it Differential Geometry, Lie
Groups and Symmetric Spaces}, Academic Press, 1978.

\bibitem{Li}{ A. Lichnerowicz}, {\it Th\'eorie globale des connexions et des groupes d'holonomie}, Edizioni Cremonese, Roma (1962).



\bibitem{mpps} C. Maclaughlin, H. Pedersen, Y.S.Poon and S. Salamon,
{\it Deformation of 2-step nilmanifolds with abelian complex
structures}, J. London Math. Soc. (2) {\bf 73} (2006), 173--193.

\bibitem{Mil} J. Milnor, {\em Curvature of left invariant metrics on
Lie groups}, Adv. Math. {\bf 21} (1976), 293-329.


\bibitem{Pe} A. P. Petravchuk, {\it Lie algebras decomposable into a sum of an abelian and a
nilpotent subalgebra},  Ukr. Math. J. {\bf 40} (3) (1988),
331--334.




\bibitem{u} L. Ugarte, {\it Hermitian structures on six-dimensional nilmanifolds}, Transformation Groups {\bf 12} (1),
(2007), 175--202.


\bibitem{V} M. Verbitsky, {\it Hypercomplex manifolds with trivial canonical bundle and their
holonomy}, Moscow Seminar on Mathematical Physics. II, 203--211,
Amer. Math. Soc. Transl. Ser. 2  {\bf 221}, Amer. Math. Soc.,
Providence, RI, 2007.


\end{thebibliography}
\end{document}